\documentclass[12pt]{article}%
\usepackage{amsmath}
\usepackage{amsfonts}
\usepackage{amssymb}
\usepackage{graphicx}%
\providecommand{\U}[1]{\protect\rule{.1in}{.1in}}
\newtheorem{theorem}{Theorem}

\newenvironment{proof}[1][Proof]{\noindent\textbf{#1.} }{\ \rule{0.5em}{0.5em}}
\begin{document}

\title{Comments on "Condensation of the digraph associated with a reciprocal matrix
and a vector", arXiv:2607.10279}
\author{Susana Furtado \thanks{sbf@fep.up.pt. CEMS.UL and Faculdade de Economia e
Gest\~{a}o da Universidade do Porto.}
\and Charles R. Johnson \thanks{crjmatrix@gmail.com. }}
\maketitle

\begin{abstract}
Since manuscript \cite{R}, recently posted on arXiv, is closely related to our
previous work, we would like to clarify some facts about it. Namely, we would
like to point out that several results in \cite{R} are already known (Theorem
1 and Corollaries 2, 3 and 4) and some observations regarding our work
\cite{FJ} are not correct. We also notice that R\'{e}dei's Theorem, a
classical result in graph theory, gives a simple tool to show Theorems 5 and 9.

\end{abstract}

We use the notation in \cite{R}. Namely, we denote by $\mathcal{PC}_{n}$ the
set of $n$-by-$n$ reciprocal matrices. Given $A=[a_{ij}]\in\mathcal{PC}_{n}$
and a positive $n$-vector $w$, we denote by $G_{A,w}$ the directed graph with
vertex set $\{1,\ldots,n\}$ and an edge $i\rightarrow j$, $i\neq j$, if and
only if $w_{i}\geq a_{ij}w_{j}.$

\bigskip

A useful known result is that a positive $n$-vector $w$ is efficient for
$A\in\mathcal{PC}_{n}$ if and only if $G_{A,w}$ is strongly connected or,
equivalently, $G_{A,w}$ has a Hamiltonian cycle. The latter equivalence
follows from results in classical graph theory, and was not shown in
\cite{R1}, as claimed in \cite{R}.

\bigskip

Theorem 1 in \cite{R} is an immediate consequence of the simple, more general,
result given next. This latter result already appears in the note \cite{R2} by
R. Fernandes. It also follows trivially from Lemma 2 in \cite{FJ} (though not
stated explicitly in that paper). In fact, after publication of \cite{FJ} and
motivated by an unfortunate sentence in the conclusions of this paper, R.
Fernandes stated it explicitly in \cite{R2}, not noticing that its proof was
implicit in \cite{FJ}.

\bigskip

If $A\in\mathcal{PC}_{n}$ is consistent and $w$ is the Perron vector of $A$,
then $G_{A,w}$ is strongly connected. So, we assume that $A$ is inconsistent.

\begin{theorem}
Let $A\in\mathcal{PC}_{n}$, $n\geq3$, be inconsistent and let $w$ be its
Perron vector$.$ Then, for any $i\in\{1,\ldots,n\}$, there is a $k\in
\{1,\ldots,n\}$ such that $k\rightarrow i$ is an edge in $G_{A,w}$ and
$i\rightarrow k$ is not.
\end{theorem}

\begin{proof}
For simplicity, using Lemma 2 in \cite{FJ} and standard arguments, we may
assume, without loss of generality, that $A$ has constant row sums $\lambda$
(the Perron eigenvalue of $A$). Then, $w$ has all entries $1$. If both
$k\rightarrow i$ and $i\rightarrow k$ are edges in $G_{A,w}$ then
\[
1=\frac{w_{i}}{w_{k}}=a_{ik}=a_{ki}.
\]
If $k\rightarrow i$ is not an edge in $G_{A,w},$ then,
\[
1=\frac{w_{k}}{w_{i}}<a_{ki}=\frac{1}{a_{ik}}.
\]
Suppose that, for any $k\in\{1,\ldots,n\}\backslash\{i\}$, either
$k\rightarrow i$ is not an edge in $G_{A,w}$ or both $k\rightarrow i$ and
$i\rightarrow k$ are edges in $G_{A,w}.$ Then,
\[
\lambda=\Sigma_{k=1}^{n}a_{ik}\leq n,
\]
a contradiction since the Perron eigenvalue $\lambda$ of the inconsistent $A$
is $>n$ (see, for example, \cite{FJ}).
\end{proof}

\bigskip

Corollary 2 in \cite{R} is an immediate consequence of the above Theorem 1 and
was already mentioned in \cite{R2}. Corollary 3 is very well-known (see, for
example the introduction of \cite{FJ}, among several other references).
Corollary 4 is Corollary 29 in \cite{FJ}, which is stated there as an "if and
only if" claim. Unfortunately, this citation does not appear in the manuscript
\cite{R} associated with the mentioned corollary.

\bigskip

Theorem 5 in \cite{R} is an immediate consequence of R\'{e}dei's Theorem for
semicomplete digraphs, which ensures the existence of a path $\gamma_{1}%
\cdots\gamma_{n}$ with $n$ vertices in $G_{A,w},$ for any $A\in\mathcal{PC}%
_{n}$ and any positive $n$-vector. Changing the pair of entries of $A$ in
positions $\gamma_{n},\gamma_{1}$ and $\gamma_{1},\gamma_{n}$ to positive
entries $a\leq\frac{w_{\gamma_{n}}}{w_{\gamma_{1}}}$ and $\frac{1}{a},$
respectively, extends the path to a Hamiltonian cycle in the graph $G_{B,w}$
of the new matrix $B$. Of course, the choice of the new entries that minimize
the Frobenius norm of $B-ww^{-1}$ is $a=\frac{w_{\gamma_{n}}}{w_{\gamma_{1}}}%
$. This choice reduces the square of the Frobenius norm of $A-ww^{-1}$ by
$(a_{\gamma_{1}\gamma_{n}}-\frac{w_{\gamma_{1}}}{w_{\gamma_{n}}}%
)^{2}+(a_{\gamma_{n}\gamma_{1}}-\frac{w_{\gamma_{n}}}{w_{\gamma_{1}}})^{2}.$
Considering all the existent paths as above, we can produce the updated matrix
$B$ that leads to the minimal Frobenius norm of $B-ww^{-1}$. These simple
observations seem to be the purpose of Theorem 7 in \cite{R}.

\bigskip

The comment just after the proof of Theorem 5 is not correct, since changing
one pair of entries in a reciprocal matrix changes the Perron vector of that
matrix as well.

\bigskip

Before Theorem 9 in \cite{R}, we have " On the other hand, Furtado and Johnson
(2024b, Theorem 33) proved that every reciprocal matrix with an inefficient
Perron vector admits an extension whose Perron vector is efficient. However,
their result does not address whether the first $n$ components of the
efficient vector of the extension can coincide with the Perron vector of the
original matrix. In this section, we use the condensation digraph to obtain
such extensions. More precisely, starting from a reciprocal matrix with an
inefficient Perron vector, we construct an extension of the reciprocal matrix
and an extension of the Perron vector, efficient for this matrix."

\bigskip

This comment about \cite{FJ} is hopelessly flawed. Theorem 9 of \cite{R}
considers a trivial question, completely different from the nontrivial problem
studied in \cite{FJ}. In fact, in Theorem 9 the extension of vector $w$ is
not, in general, the Perron vector of the extended reciprocal matrix, while in
\cite{FJ} the matrix is extended to have an efficient Perron vector. (By the
way, the second claim in Theorem 9 doesn't seem well stated, as $w^{\prime}$
is related with $B$). Moreover, in Theorem 9, $w$ being the Perron vector of
$A$ is not important. In fact, the next more general result, which has a very
simple proof, is valid. Given $B\in\mathcal{PC}_{n+1},$ we denote by $B(n+1)$
the principal submatrix of $B$ obtained by deleting row and column $n+1$.

\begin{theorem}
Let $A=[a_{ij}]\in\mathcal{PC}_{n}$, $n\geq2.$ Let $w=\left[
\begin{array}
[c]{ccc}%
w_{1} & \cdots & w_{n}%
\end{array}
\right]  ^{T}$ be a positive $n$-vector. Then, there are $B\in\mathcal{PC}%
_{n+1},$ with $B(n+1)=A,$ and $w_{n+1}>0$ such that $w^{\prime}=\left[
\begin{array}
[c]{c}%
w\\
w_{n+1}%
\end{array}
\right]  $ is efficient for $B.$ Moreover, $B$ can be chosen so that the
corresponding $w_{n+1}$ is arbitrarily small.
\end{theorem}

\begin{proof}
By R\'{e}dei's Theorem for semicomplete digraphs, there is a path with $n$
vertices in $G_{A,w.}$ Without loss of generality, suppose that $12\cdots n$
is such a path. This means that $\frac{w_{i}}{w_{i+1}}\geq a_{i,i+1},$ for
$i=1,\ldots,n-1.$ Now choose $a_{n,n+1},$ $a_{n+1,1}$ and $w_{n+1}$ such that%
\[
\frac{w_{n}}{w_{n+1}}\geq a_{n,n+1}>0\text{ and }\frac{w_{n+1}}{w_{1}}\geq
a_{n+1,1}>0.
\]
Choose the remaining entries $a_{n+1,j}>0,$ $j=2,\ldots,n-1$, arbitrarily.
Then the reciprocal matrix $B=[a_{ij}]\in\mathcal{PC}_{n+1}$ extends $A$.
Moreover, $w^{\prime}$ is efficient for $B$, as $12\cdots n(n+1)1$ is a
Hamiltonian cycle in $G_{B,w^{\prime}.}$ Of course, by a convenient choice of
$a_{n,n+1}$ and $a_{n+1,1}$, $w_{n+1}$ can be chosen arbitrarily small.
\end{proof}

\bigskip

\bigskip

\bigskip

\bigskip


\bigskip

\begin{thebibliography}{9}                                                                                                %


\bibitem {R1}R. Fernandes, \textit{R. Palma, Positive vectors, pairwise
comparison matrices and directed Hamiltonian cycles}, Linear Algebra and its
Applications 699, 312-330 (2024).

\bibitem {R2}R. Fernandes, \textit{On a conjecture concerning the extensions
of a reciprocal matrix}, arXiv:2507.16593 [math.CO].

\bibitem {R}R. Fernandes, \textit{Condensation of the digraph associated with
a reciprocal matrix and a vector}, arXiv:2607.10279 [math.CO].

\bibitem {FJ}S. Furtado, C. R. Johnson, \textit{Efficiency analysis for the
Perron vector of a reciprocal matrix}, Applied Mathematics and Computation 480
(2024), 128913.
\end{thebibliography}
\end{document}